\documentclass[9pt]{article}%
\usepackage{bbm}
\usepackage{epsfig}
\usepackage{easybmat}
\usepackage{graphics,graphicx,amssymb,amsmath,verbatim}
\usepackage{amssymb}
\usepackage{mathrsfs}
\usepackage{amsfonts}
\usepackage{amsmath}
\usepackage{graphicx}
\usepackage{easybmat}%
\providecommand{\U}[1]{\protect \rule{.1in}{.1in}}
\newtheorem{theorem}{Theorem}

\newtheorem{corollary}{Corollary}

\newtheorem{lemma}{Lemma}

\newtheorem{proposition}{Proposition}
\newtheorem{remark}{Remark}

\newenvironment{proof}[1][Proof]{\noindent \textbf{#1.} }{\  \rule{0.5em}{0.5em}}
\allowdisplaybreaks[4]
\begin{document}
\title{Stability Analysis of Integral Delay Systems with Multiple Delays}
\author{Bin Zhou\thanks{Bin Zhou is with the Center for Control Theory and Guidance Technology, Harbin
Institute of Technology, Harbin, 150001, China. Email:
\texttt{binzhou@hit.edu.cn; binzhoulee@163.com. }} \quad \quad \quad Zhao-Yan Li\thanks{Zhao-Yan Li is with the Department of Mathematics, Harbin Institute
of Technology, Harbin, 150001, China. Email: \texttt{zhaoyanlee@gmail.com;
lizhaoyan@hit.edu.cn.}}}
\date{}
\maketitle
\begin{abstract}
This note is concerned with stability analysis of integral delay systems with
multiple delays. To study this problem, the well-known Jensen inequality is
generalized to the case of multiple terms by introducing an individual slack
weighting matrix for each term, which can be optimized to reduce the
conservatism. With the help of the multiple Jensen inequalities and by
developing a novel linearizing technique, two novel Lyapunov functional based
approaches are established to obtain sufficient stability conditions expressed
by linear matrix inequalities (LMIs). It is shown that these new conditions
are always less conservative than the existing ones. Moreover, by the positive
operator theory, a single LMI based condition and a spectral radius based
condition are obtained based on an existing sufficient stability condition
expressed by coupled LMIs. A numerical example illustrates the effectiveness
of the proposed approaches.

\end{abstract}

\section{Introduction}

An integral delay system (IDS) in the form of $x\left(  t\right)  =\int
_{-\tau}^{0}F\left(  s\right)  x\left(  t+s\right)  \mathrm{d}s,$ where
$F\left(  s\right)  \ $is a matrix function with bounded variation, has many
important applications in the study of time-delay systems (see, for example,
{\cite{gil14sma}}, \cite{hale77book,zld12auto} and \cite{zhou14book}). In
Hale's book \cite{hale77book} this class of IDSs were named as $D$ operators,
which were also treated as generalized difference equations there, and their
stability is {necessary} for the stability of the associated neutral
time-delay systems. This class of IDSs also come from the model reduction
approach for stability analysis of time-delay systems, which are frequently
named as the additional dynamics (see \cite{gn01tac} and \cite{km03tac}) and
their stability is necessary for the stability of the transformed time-delay
systems. This class of IDSs are very closely related with the predictor
feedback control of linear systems with input delays, for example,

\begin{itemize}
\item In \cite{gu12isrn}, the author proved that the numerical implementation
of the predictor feedback for linear systems with input delays is safe only if
an IDS is exponentially stable.

\item In \cite{lzl14auto}\ it is shown that the stability of this class of
IDSs is necessary for the robust stability of linear systems with input delay
by the well-known predictor feedback.

\item In \cite{zhou14auto} we have shown that an input delayed linear system
by the so-called pseudo-predictor feedback is exponentially stable if and only
if an IDS is exponentially stable.
\end{itemize}

Stability analysis of IDSs can be traced back at least to Cruz and Hale
\cite{rh70jde}, Henry \cite{herry70jde} and Melvin \cite{melvn74jaa}, in the
study of the stability of neutral time-delay systems \cite{hale77book}. When
the right hand side of the IDS only contains terms at some time points, a
general theory was build in \cite{carvalho96laa}. This class of IDSs have
received renewed interest in recent years. A general stability theorem was
build in \cite{dll10ijrnc} and was later applied on different forms of IDSs
(see \cite{lzl13auto}, \cite{Aguliar10AMC}, \cite{mm12tac}, and the references
therein). {In the paper \cite{gil14book} stability conditions are
derived for IDS with matrix discrete-continuous measures.}

In this note we also study stability analysis of this class of IDSs in which
$F\left(  s\right)  $ is a piecewise constant matrix function, namely, the
system can be expressed as $x\left(  t\right)  =%
{\textstyle \sum \nolimits_{i=1}^{N}}
A_{i}\int_{-\tau_{i}}^{0}x\left(  t+s\right)  \mathrm{d}s$ where $A_{i}$ and
$\tau_{i}$ are constants ($\tau_{i}$ can be unknown, time-varying, yet
bounded). This class of IDSs with multiple time delays have been investigated
in \cite{Aguliar10AMC} with the help of the well-known discrete-time and
continuous-time Jensen inequalities. In this note, by recognizing that the
jointed usage of the conventional discrete-time and continuous-time Jensen
inequalities requires that all the integration terms $\int_{-\tau_{i}}%
^{0}x^{\mathrm{T}}\left(  t+s\right)  A_{i}^{\mathrm{T}}QA_{i}x\left(
t+s\right)  \mathrm{d}s$ share the same weighting matrix $Q,$ we first
establish a so-called multiple Jensen inequality, by which as well as some
novel Lyapunov functionals and a new linearization technique, every individual
integration term possesses a different weighting matrix, which can introduce
more weighting matrices that will be optimized to reduce the conservatism of
the corresponding sufficient stability conditions. Indeed, it is shown in both
theory and by numerical examples that the stability conditions obtained by the
multiple Jensen inequalities are always less conservative than that obtained
by the jointed usage of the conventional discrete-time and continuous-time
Jensen inequalities. Another contribution of this note is that, with the help
of the positive operator theory, we are able to establish an equivalent linear
matrix inequality (LMI) based stability condition involving only \textit{one
constraint} and \textit{one decision matrix} and an equivalent spectral radius
based stability condition of some existing stability conditions that are
expressed by \textit{a set of coupled LMIs}. These two equivalent stability
conditions are appealing in both theory and in computation.

The remaining of this note is organized as follows. The problem formulation
and some preliminaries are given in Section \ref{sec1}. In Section \ref{sec2},
two kinds of LMIs based sufficient conditions are established with the help of
the multiple Jensen inequalities, some novel Lyapunov functionals and a new
linearization technique. A spectral radius based stability condition and a
single LMI based stability condition are then established in Section
\ref{sec3} in which\ a comparison of the proposed approaches and the existing
one will also be carried out. A numerical example is worked out in Section
\ref{sec4} to illustrate the effectiveness of the proposed approaches and
finally Section \ref{sec5} concludes the note.

\section{\label{sec1}Problem Formulation and Preliminaries}

Consider the following integral delay system (IDS) with multiple delays%
\begin{equation}
x\left(  t\right)  =\sum \limits_{i=1}^{N}A_{i}\int_{-\tau_{i}}^{0}x\left(
t+s\right)  \mathrm{d}s, \label{ids}%
\end{equation}
where $A_{i}\in \mathbf{R}^{n\times n},i\in \mathbf{I}\left[  1,N\right]
\triangleq \{1,2,\ldots,N\},$ are given matrices and $\tau_{i},i\in
\mathbf{I}\left[  1,N\right]  ,$ are given scalars and are such that{
\begin{equation}
0<\tau_{i}\leq \tau \triangleq \max_{i\in \mathbf{I}\left[  1,N\right]  }\{
\tau_{i}\},\;i\in \mathbf{I}\left[  1,N\right]. \label{eq11}%
\end{equation}}
Let $\varphi \in \mathscr{C}_{n,\tau}$ be an
initial condition for (\ref{ids}) and $x\left(  t\right)  =x\left(
t,\varphi \right)  ,\forall t\geq0$ be the corresponding solution of
(\ref{ids}) satisfying $x\left(  t\right)  =\varphi \left(  t\right)  ,\forall
t\in {\lbrack-\tau,0)}.$ Here $\mathscr{C}_{n,\tau}$ denotes the Banach space of
continuous vector functions mapping the interval $\left[  -\tau,0\right]  $
into $\mathbf{R}^{n}$ with the topology of uniform convergence. We say that
the IDS (\ref{ids}) is exponentially stable if there exist two positive
constants $\alpha$ and $\beta$ such that $\left \Vert x\left(  t\right)
\right \Vert \leq \alpha \sup_{s\in {\lbrack-\tau,0)}}\left \Vert \varphi \left(
s\right)  \right \Vert \mathrm{e}^{-\beta t},\forall t\geq0.$

{The IDS (\ref{ids}) arises when some transformations are made on
differential-difference systems \cite{Aguliar10AMC}}. In this note we
are concerned with the stability analysis of the IDS (\ref{ids}). By choosing
some suitable Lyapunov functionals and developing a new linearization
technique for handling nonlinear matrix inequalities, we will establish two
classes of LMIs based sufficient conditions guaranteeing the exponential
stability of the IDS (\ref{ids}). Moreover, with the help of the positive
operator theory, we will also provide a spectral radius based sufficient
stability condition. The relationships among these different sufficient
conditions are also revealed. {Our results improve those in
\cite{Aguliar10AMC}}. Both theoretical analysis and numerical examples will
demonstrate that the obtained results are always less conservative and more
efficient than the existing ones {especially those in
\cite{Aguliar10AMC}}.

The following general Lyapunov stability theorem for the IDS (\ref{ids}) will
be used later in this note.

\begin{lemma}
\label{lm0}\cite{Aguliar10AMC} The IDS (\ref{ids}) is exponentially stable if
there exists a differentiable functional $V:\mathscr{C}_{n,\tau}%
\rightarrow \mathbf{R}$ and three positive constants $\alpha_{i},i=1,2,3,$ such
that
\begin{align}
\alpha_{1}\int_{-\tau}^{0}\left \Vert x\left(  t+\theta \right)  \right \Vert
^{2}\mathrm{d}\theta &  \leq V\left(  x_{t}\right)  \leq \alpha_{2}\int_{-\tau
}^{0}\left \Vert x\left(  t+\theta \right)  \right \Vert ^{2}\mathrm{d}%
\theta,\label{eq64}\\
\dot{V}\left(  x_{t}\right)   &  \leq-\alpha_{3}\int_{-\tau}^{0}\left \Vert
x\left(  t+\theta \right)  \right \Vert ^{2}\mathrm{d}\theta. \label{eq65}%
\end{align}

\end{lemma}

At the end of this section, we give the following technical lemma which is
helpful for the linearization of nonlinear matrix inequalities in the sequel.

\begin{lemma}
\label{lm2}Let $S\in \mathbf{R}^{n\times n}$ and $Q\in \mathbf{R}^{n\times n}$
be two positive definite matrices. Then $Q<S^{-1}$ if and only if there exists
a matrix $R\in \mathbf{R}^{n\times n}$ such that%
\begin{equation}
R^{\mathrm{T}}QR+S-\left(  R+R^{\mathrm{T}}\right)  <0. \label{eq80}%
\end{equation}
The same statements hold true if \textquotedblleft%
$<$%
\textquotedblright \ in the above two inequalities are replaced by
\textquotedblleft$\leq$\textquotedblright.
\end{lemma}

\begin{proof}
It follows from (\ref{eq80}) that $R$ is nonsingular. Since $\left(
R-S\right)  ^{\mathrm{T}}S^{-1}\left(  R-S\right)  \geq0,$ namely,
\begin{equation}
-R^{\mathrm{T}}S^{-1}R\leq S-\left(  R+R^{\mathrm{T}}\right)  , \label{eq61}%
\end{equation}
we have from (\ref{eq80}) that
\begin{equation}
Q<-R^{-\mathrm{T}}(S-(R+R^{\mathrm{T}}))R^{-1}{\leq} R^{-\mathrm{T}%
}R^{\mathrm{T}}S^{-1}RR^{-1}=S^{-1}. \label{eq69}%
\end{equation}
On the other hand, if $Q<S^{-1}$ is satisfied, then (\ref{eq80}) is satisfied
by choosing $R=S$.
\end{proof}

\section{\label{sec2}The Multiple Jensen Inequality Based Stability
Conditions}

\subsection{The Multiple Jensen Inequality}

We first recall the following well-known Jensen inequality.

\begin{lemma}
\cite{Gu00cdc} For any positive definite matrix $Q>0,$ a positive number
$\tau>0,$ and a vector valued function $\omega:\left[  -\tau,0\right]
\rightarrow \mathbf{R}^{n}$ such that the integrals in the following are
well-defined, then%
\begin{equation}
\left(  \int_{-\tau}^{0}\omega \left(  s\right)  \mathrm{d}s\right)
^{\mathrm{T}}Q\left(  \int_{-\tau}^{0}\omega \left(  s\right)  \mathrm{d}%
s\right)  \leq \tau \int_{-\tau}^{0}\omega^{\mathrm{T}}\left(  s\right)
Q\omega \left(  s\right)  \mathrm{d}s. \label{cjensen}%
\end{equation}
Moreover, for a series of vectors $\xi_{i}\in \mathbf{R}^{n},i\in
\mathbf{I}\left[  1,N\right]  ,$ there holds%
\begin{equation}
\left(  \sum \limits_{i=1}^{N}\xi_{i}\right)  ^{\mathrm{T}}Q\left(
\sum \limits_{i=1}^{N}\xi_{i}\right)  \leq N\sum \limits_{i=1}^{N}\xi
_{i}^{\mathrm{T}}Q\xi_{i}. \label{djensen}%
\end{equation}

\end{lemma}

Inequalities (\ref{cjensen}) and (\ref{djensen}) are respectively known as the
continuous-time Jensen inequality and the discrete-time Jensen inequality,
which have been widely used in the literature for the stability analysis and
stabilization of time-delay systems (see \cite{Gu00cdc} and the references
that have cited it). By using these two inequalities jointly we get the
following corollary.

\begin{corollary}
\label{coro3}Let $\tau_{i}\geq0,i\in \mathbf{I}\left[  1,N\right]  ,$ be $N$
given nonnegative scalars. Assume that $\omega_{i}:\left[  -\tau_{i},0\right]
\rightarrow \mathbf{R}^{n},i\in \mathbf{I}\left[  1,N\right]  ,$ are such that
the integrals in the following are well-defined, then%
\begin{equation}
\left(  \sum \limits_{i=1}^{N} x_{i}\right)  ^{\mathrm{T}}Q\left(
\sum \limits_{i=1}^{N}x_{i}\right)  \leq N\sum \limits_{i=1}^{N}\tau_{i}%
\int_{-\tau_{i}}^{0}\omega_{i}^{\mathrm{T}}\left(  s\right)  Q\omega
_{i}\left(  s\right)  \mathrm{d}s, \label{cdjensen}%
\end{equation}
where $x_{i}=\int_{-\tau_{i}}^{0}\omega_{i}\left(  s\right)  \mathrm{d}%
s,i\in \mathbf{I}\left[  1,N\right]  .$
\end{corollary}

We notice that all the $N$ integrations $\int_{-\tau_{i}}^{0}\omega
_{i}^{\mathrm{T}}\left(  s\right)  Q\omega_{i}\left(  s\right)  \mathrm{d}%
s,i\in \mathbf{I}\left[  1,N\right]  ,$ on the right hand side of
(\ref{cdjensen}) share the \textit{same} weighting matrix $Q,$ which
is clearly very restrictive. To reduce the possible conservatism, we introduce
the following multiple Jensen inequality.

\begin{lemma}
\label{lmmjensen}Let $Q_{i}\in \mathbf{R}^{n\times n},i\in \mathbf{I}\left[
1,N\right]  ,$ be $N$ given positive definite matrices and $\tau_{i}%
>0,i\in \mathbf{I}\left[  1,N\right]  ,$ be $N$ given
scalars. Assume that the vector functions $\omega_{i}:\left[  -\tau
_{i},0\right]  \rightarrow \mathbf{R}^{n},i\in \mathbf{I}\left[  1,N\right]  ,$
are such that the integrals in the following are well-defined, then%
\begin{equation}
\left(  \sum \limits_{i=1}^{N}x_{i}\right)  ^{\mathrm{T}}{Q^{-1}}\left(
\sum \limits_{i=1}^{N}x_{i}\right)  \leq \sum \limits_{i=1}^{N}\int_{-\tau_{i}%
}^{0}\omega_{i}^{\mathrm{T}}\left(  s\right)  \tau_{i}Q_{i}^{-1}\omega
_{i}\left(  s\right)  \mathrm{d}s, \label{jensen}%
\end{equation}
where $x_{i}=\int_{-\tau_{i}}^{0}\omega_{i}\left(  s\right)  \mathrm{d}%
s,i\in \mathbf{I}\left[  1,N\right]  $ and {$Q=%
{\textstyle \sum \nolimits_{i=1}^{N}}
Q_{i}.$} Moreover, for a series of vectors $\xi_{i}\in \mathbf{R}^{n}%
,i\in \mathbf{I}\left[  1,N\right]  ,$ there holds%
\begin{equation}
\left(  \sum \limits_{i=1}^{N}\xi_{i}\right)  ^{\mathrm{T}}\left(
\sum \limits_{i=1}^{N}Q_{i}\right)  ^{-1}\left(  \sum \limits_{i=1}^{N}\xi
_{i}\right)  \leq \sum \limits_{i=1}^{N}\xi_{i}^{\mathrm{T}}Q_{i}^{-1}\xi_{i}.
\label{mdjensen}%
\end{equation}

\end{lemma}

\begin{proof}
Notice that, for any $i\in \mathbf{I}\left[  1,N\right]  ,$ by a Schur
complement, there holds%
\begin{equation}
\left[
\begin{array}
[c]{cc}%
\omega_{i}^{\mathrm{T}}\left(  s\right)  \tau_{i}Q_{i}^{-1}\omega_{i}\left(
s\right)  & \omega_{i}^{\mathrm{T}}\left(  s\right) \\
\omega_{i}\left(  s\right)  & \frac{1}{\tau_{i}}Q_{i}%
\end{array}
\right]  \geq0,\;i\in \mathbf{I}\left[  1,N\right]  . \label{eq2}%
\end{equation}
Taking integration on both sides of the above inequality gives%
\begin{equation}
\left[
\begin{array}
[c]{cc}%
\int_{-\tau_{i}}^{0}\omega_{i}^{\mathrm{T}}\left(  s\right)  \tau_{i}%
Q_{i}^{-1}\omega_{i}\left(  s\right)  \mathrm{d}s & \int_{-\tau_{i}}^{0}%
\omega_{i}^{\mathrm{T}}\left(  s\right)  \mathrm{d}s\\
\int_{-\tau_{i}}^{0}\omega_{i}\left(  s\right)  \mathrm{d}s & Q_{i}%
\end{array}
\right]  \geq0, \label{eq3}%
\end{equation}
where $i\in \mathbf{I}\left[  1,N\right]  $ , which implies%
\begin{equation}
\left[
\begin{array}
[c]{cc}%
\sum \limits_{i=1}^{N}\int_{-\tau_{i}}^{0}\omega_{i}^{\mathrm{T}}\left(
s\right)  \tau_{i}Q_{i}^{-1}\omega_{i}\left(  s\right)  \mathrm{d}s &
\sum \limits_{i=1}^{N}\int_{-\tau_{i}}^{0}\omega_{i}^{\mathrm{T}}\left(
s\right)  \mathrm{d}s\\
\sum \limits_{i=1}^{N}\int_{-\tau_{i}}^{0}\omega_{i}\left(  s\right)
\mathrm{d}s & \sum \limits_{i=1}^{N}Q_{i}%
\end{array}
\right]  \geq0. \label{eq4}%
\end{equation}
By a Schur complement again, (\ref{eq4}) is equivalent to
(\ref{jensen}). Finally, the inequality in (\ref{mdjensen}) can be proven in a similar way.
\end{proof}

Now every integration $\int_{-\tau_{i}}^{0}\omega_{i}^{\mathrm{T}}\left(
s\right)  Q_{i}\omega_{i}\left(  s\right)  \mathrm{d}s,i\in \mathbf{I}\left[
1,N\right]  ,$ on the right hand side of (\ref{jensen}) is weighted by an
individual weighting matrix $Q_{i},i\in \mathbf{I}\left[  1,N\right]  ,$ which
can introduce more decision variables that can be optimized to reduce the
conservatism of the resulting conditions. The multiple Jensen inequalities
(\ref{jensen}) and (\ref{mdjensen}) are clearly less conservative than the
inequalities in (\ref{cdjensen}) and (\ref{djensen}) since the later ones can
be obtained immediately by setting $Q_{i}=Q^{-1},i\in \mathbf{I}\left[
1,N\right]  ,$ in the former ones.

By applying the Jensen inequality (\ref{cdjensen}) in Corollary \ref{coro3}
and choosing the following Lyapunov functional{
\begin{align}
V\left(  x_{t}\right)   &  =\int_{t-\tau}^{t}x^{\mathrm{T}}\left(  s\right)
Px\left(  s\right)  \mathrm{d}s\nonumber \\
&  +\sum \limits_{i=1}^{N}\int_{-\tau_{i}}^{0}\left(  s+\tau_{i}\right)
x^{\mathrm{T}}\left(  t+s\right)  Q_{i}x\left(  t+s\right)  \mathrm{d}s,
\label{eqv2}%
\end{align}}
the following result was obtained in \cite{Aguliar10AMC}.

\begin{lemma}
\label{lm1}\cite{Aguliar10AMC} The IDS (\ref{ids}) is exponentially stable if
there exist $N+1$ positive definite matrices $P,Q_{i}\in \mathbf{R}^{n\times
n},i\in \mathbf{I}\left[  1,N\right]  ,$ such that the following coupled LMIs
are satisfied%
\begin{equation}
N\tau_{i}A_{i}^{\mathrm{T}}\left(  P+\sum \limits_{j=1}^{N}\tau_{j}%
Q_{j}\right)  A_{i}-Q_{i}<0,\;i\in \mathbf{I}\left[  1,N\right]  .
\label{LMIamc}%
\end{equation}

\end{lemma}

In the next two subsections, we will show how to use the multiple Jensen
inequality (\ref{jensen}) to improve the above result.

\subsection{The First Sufficient Stability Condition}

In this subsection we present a new sufficient condition for the exponential
stability of the IDS (\ref{ids}) by applying the multiple Jensen inequality
(\ref{jensen}) and choosing a similar Lyapunov functional as (\ref{eqv2}).

\begin{theorem}
\label{th3}Consider the IDS (\ref{ids}). Then

\begin{enumerate}
\item It is exponentially stable if there exist $2N$ positive definite
matrices $S_{i},Q_{i}\in \mathbf{R}^{n\times n},i\in \mathbf{I}\left[
1,N\right]  ,$ such that the following nonlinear matrix inequalities are
satisfied
\begin{align}
&  \tau_{i}^{2}A_{i}^{\mathrm{T}}Q_{i}^{-1}A_{i}-S_{i}<0,\;i\in \mathbf{I}%
\left[  1,N\right]  ,\label{NMI1}\\
&  \sum \limits_{i=1}^{N}S_{i}<\left(  \sum \limits_{i=1}^{N}Q_{i}\right)
^{-1}. \label{NMI2}%
\end{align}

\item The nonlinear matrix inequalities in (\ref{NMI1})--(\ref{NMI2}) are
solvable if and only if there exist $2N$ positive definite matrices
$S_{i},Q_{i}\in \mathbf{R}^{n\times n},i\in \mathbf{I}\left[  1,N\right]  ,$ and
a matrix $R\in \mathbf{R}^{n\times n}$ such that the following LMIs are
satisfied%
\begin{align}
&  \sum \limits_{i=1}^{n}Q_{i}+\sum \limits_{i=1}^{n}S_{i}-\left(
R^{\mathrm{T}}+R\right)  <0,\label{LMI1}\\
&  \left[
\begin{array}
[c]{cc}%
-S_{i} & \tau_{i}A_{i}^{\mathrm{T}}R\\
\tau_{i}R^{\mathrm{T}}A_{i} & -Q_{i}%
\end{array}
\right]  <0,\;i\in \mathbf{I}\left[  1,N\right]  . \label{LMI2}%
\end{align}

\end{enumerate}
\end{theorem}

\begin{proof}
\textit{Proof of Item 1.} The inequality in (\ref{NMI2}) implies that there
exists a positive definite matrix $P>0$ such that $P\leq \varepsilon I_{n}$ for
some sufficiently small number $\varepsilon>0$ and such that%
\begin{equation}
P+\sum \limits_{i=1}^{N}S_{i}<\left(  \sum \limits_{i=1}^{N}Q_{i}\right)
^{-1},\;i\in \mathbf{I}\left[  1,N\right]  . \label{eq45}%
\end{equation}
Consider the following Lyapunov functional{
\begin{align}
V\left(  x_{t}\right)   &  =\int_{t-\tau}^{t}x^{\mathrm{T}}\left(  s\right)
Px\left(  s\right)  \mathrm{d}s\nonumber \\
&  +\sum \limits_{i=1}^{N}\int_{-\tau_{i}}^{0}\left(  \frac{s}{\tau_{i}%
}+1\right)  x^{\mathrm{T}}\left(  t+s\right)  S_{i}x\left(  t+s\right)
\mathrm{d}s, \label{eq46}%
\end{align}}
which is in the form of (\ref{eqv2}) {that was used in \cite{Aguliar10AMC}.} The
time-derivative of $V\left(  x_{t}\right)  $ satisfies
\begin{align}
\dot{V}\left(  x_{t}\right)   &  \leq x^{\mathrm{T}}\left(  t\right)  \left(
P+\sum \limits_{i=1}^{N}S_{i}\right)  x\left(  t\right)  -\sum \limits_{i=1}%
^{N}y_{i}\nonumber \\
&  \leq x^{\mathrm{T}}\left(  t\right)  \left(  \sum \limits_{i=1}^{N}%
Q_{i}\right)  ^{-1}x\left(  t\right)  -\sum \limits_{i=1}^{N}y_{i}\nonumber \\
&  \leq \sum \limits_{i=1}^{N}\int_{-\tau_{i}}^{0}x^{\mathrm{T}}\left(
t+s\right)  \tau_{i}A_{i}^{\mathrm{T}}Q_{i}^{-1}A_{i}x\left(  t+s\right)
\mathrm{d}s-\sum \limits_{i=1}^{N}y_{i}\nonumber \\
&  =\sum \limits_{i=1}^{N}\frac{1}{\tau_{i}}\int_{-\tau_{i}}^{0}x^{\mathrm{T}%
}\left(  t+s\right)  \left(  \tau_{i}^{2}A_{i}^{\mathrm{T}}Q_{i}^{-1}%
A_{i}-S_{i}\right)  x\left(  t+s\right)  \mathrm{d}s\nonumber \\
&  \leq-\gamma \int_{-\tau}^{0}\left \Vert x\left(  t+s\right)  \right \Vert
^{2}\mathrm{d}s, \label{eq47}%
\end{align}
where $y_{i}=\frac{1}{\tau_{i}}\int_{-\tau_{i}}^{0}x^{\mathrm{T}}\left(
t+s\right)  S_{i}x\left(  t+s\right)  \mathrm{d}s,$ $\gamma>0$ is some
constant and we have used the multiple Jensen inequality (\ref{jensen}) and
the nonlinear matrix inequalities (\ref{NMI1}). Hence, it follows from Lemma
\ref{lm0} that the IDS (\ref{ids}) is exponentially stable.

\textit{Proof of Item 2}. Let $S=%
{\textstyle \sum \nolimits_{i=1}^{N}}
S_{j}$ and $Q=%
{\textstyle \sum \nolimits_{i=1}^{N}}
Q_{j}.$ Then (\ref{NMI2}) is equivalent to $Q<S^{-1},$ which, by Lemma
\ref{lm2}, is satisfied if and only if there exists an $R\in \mathbf{R}%
^{n\times n}$ such that%
\begin{equation}
R^{\mathrm{T}}QR+S-\left(  R^{\mathrm{T}}+R\right)  <0. \label{eq40}%
\end{equation}
On the other hand, by the Schur complement, the inequalities in (\ref{NMI1})
are satisfied if and only if%
\begin{equation}
\left[
\begin{array}
[c]{cc}%
-S_{i} & \tau_{i}A_{i}^{\mathrm{T}}\\
\tau_{i}A_{i} & -Q_{i}%
\end{array}
\right]  <0,\;i\in \mathbf{I}\left[  1,N\right]  , \label{eq41}%
\end{equation}
which, by a congruence transformation, are equivalent to%
\begin{equation}
\left[
\begin{array}
[c]{cc}%
-S_{i} & \tau_{i}A_{i}^{\mathrm{T}}R\\
\tau_{i}R^{\mathrm{T}}A_{i} & -R^{\mathrm{T}}Q_{i}R
\end{array}
\right]  <0,\;i\in \mathbf{I}\left[  1,N\right]  . \label{eq42}%
\end{equation}
It is clear that (\ref{eq40}) and (\ref{eq42}) are respectively equivalent to
(\ref{LMI1}) and (\ref{LMI2}) by a substitution $R^{\mathrm{T}}Q_{i}%
R\rightarrow Q_{i},i\in \mathbf{I}\left[  1,N\right]  $ (and thus
$R^{\mathrm{T}}QR\rightarrow Q$). The proof is finished.
\end{proof}


\subsection{The Second Sufficient Stability Condition}

With the help of the multiple Jensen inequality (\ref{jensen}), we further
present in this subsection a new sufficient condition for the exponential
stability of the IDS (\ref{ids}) with an alternative Lyapunov functional,
which may possess some advantages over (\ref{eqv2}) and (\ref{eq46}).

\begin{theorem}
\label{th1}Consider the IDS (\ref{ids}). Then

\begin{enumerate}
\item It is exponentially stable if there exist $N$ positive definite matrices
$Q_{i}\in \mathbf{R}^{n\times n},i\in \mathbf{I}\left[  1,N\right]  ,$ such that
the following nonlinear matrix inequality is satisfied%
\begin{equation}
\sum \limits_{i=1}^{N}\tau_{i}^{2}A_{i}^{\mathrm{T}}Q_{i}^{-1}A_{i}-\left(
\sum \limits_{i=1}^{n}Q_{i}\right)  ^{-1}<0. \label{eq44}%
\end{equation}

\item The nonlinear matrix inequality (\ref{eq44}) is solvable if and only if
there exist $N$ positive definite matrices $Q_{i}\in \mathbf{R}^{n\times
n},i\in \mathbf{I}\left[  1,N\right]  ,$ such that the following LMI is
satisfied%
\begin{equation}%
{\displaystyle \sum \limits_{i=1}^{N}}
\left[
\begin{array}
[c]{c}%
\tau_{1}A_{1}\\
\vdots \\
\tau_{N}A_{N}%
\end{array}
\right]  Q_{i}\left[
\begin{array}
[c]{c}%
\tau_{1}A_{1}\\
\vdots \\
\tau_{N}A_{N}%
\end{array}
\right]  ^{\mathrm{T}}-\left[
\begin{array}
[c]{ccc}%
Q_{1} &  & \\
& \ddots & \\
&  & Q_{N}%
\end{array}
\right]  <0. \label{LMI}%
\end{equation}

\end{enumerate}
\end{theorem}

\begin{proof}
\textit{Proof of Item 1}. Let $Q=%
{\textstyle \sum \nolimits_{i=1}^{N}}
Q_{j}.$ Then it follows from (\ref{eq44}) that there exist two sufficiently
small numbers $\delta>0$ and $\varepsilon>0$ such that%
\begin{equation}
\sum \limits_{i=1}^{N}\left(  \tau_{i}^{2}A_{i}^{\mathrm{T}}Q_{i}^{-1}%
A_{i}+\tau_{i}\delta I_{n}\right)  -Q^{-1}\leq-\varepsilon Q^{-1}.
\label{eq28}%
\end{equation}
Let $R_{i}>0,i\in \mathbf{I}\left[  1,N\right]  ,$ be such that
\begin{equation}
\sum \limits_{i=1}^{N}R_{i}\triangleq R=Q^{-1}=\left(
{\displaystyle \sum \limits_{i=1}^{N}}
Q_{i}\right)  ^{-1}, \label{eqr}%
\end{equation}
and consider an associated nonnegative functional%
\begin{equation}
V_{1}\left(  x_{t}\right)  =\sum \limits_{i=1}^{N}\int_{t-\tau_{i}}%
^{t}x^{\mathrm{T}}\left(  s\right)  R_{i}x\left(  s\right)  \mathrm{d}s,
\label{eqv1}%
\end{equation}
whose time-derivative is given by%
\begin{align}
&  \dot{V}_{1}\left(  x_{t}\right)  =x^{\mathrm{T}}\left(  t\right)  \left(
\sum \limits_{i=1}^{N}R_{i}\right)  x\left(  t\right)  -\sum \limits_{i=1}%
^{N}y_{i}\nonumber \\
&  =x^{\mathrm{T}}\left(  t\right)  \left(
{\displaystyle \sum \limits_{i=1}^{N}}
Q_{i}\right)  ^{-1}x\left(  t\right)  -\sum \limits_{i=1}^{N}y_{i}\nonumber \\
&  \leq%
{\displaystyle \sum \limits_{i=1}^{N}}
\int_{-\tau_{i}}^{0}\tau_{i}x^{\mathrm{T}}\left(  t+s\right)  A_{i}%
^{\mathrm{T}}Q_{i}^{-1}A_{i}x\left(  t+s\right)  \mathrm{d}s-\sum
\limits_{i=1}^{N}y_{i}, \label{eq19}%
\end{align}
where $y_{i}=x\left(  t-\tau_{i}\right)  ^{\mathrm{T}}R_{i}x\left(  t-\tau
_{i}\right)  ,$ and we have used the IDS (\ref{ids}) and the multiple Jensen
inequality (\ref{jensen}).

Choose another nonnegative functional%
\begin{equation}
V_{2}\left(  x_{t}\right)  =\sum \limits_{i=1}^{N}\int_{0}^{\tau_{i}}\int
_{t-s}^{t}x^{\mathrm{T}}\left(  l\right)  \left(  \tau_{i}A_{i}^{\mathrm{T}%
}Q_{i}^{-1}A_{i}+\delta I_{n}\right)  x\left(  l\right)  \mathrm{d}%
l\mathrm{d}s, \label{v2}%
\end{equation}
whose time-derivative can be evaluated as%
\begin{align}
&  \dot{V}_{2}\left(  x_{t}\right)  =\sum \limits_{i=1}^{N}x^{\mathrm{T}%
}\left(  t\right)  \left(  \tau_{i}^{2}A_{i}^{\mathrm{T}}Q_{i}^{-1}%
A_{i}+\delta \tau_{i}I_{n}\right)  x\left(  t\right) \nonumber \\
&  -\sum \limits_{i=1}^{N}\int_{-\tau_{i}}^{0}x^{\mathrm{T}}\left(  t+s\right)
\left(  \tau_{i}A_{i}^{\mathrm{T}}Q_{i}^{-1}A_{i}+\delta I_{n}\right)
x\left(  t+s\right)  \mathrm{d}s. \label{dv2}%
\end{align}
Hence, it follows from (\ref{eq19}) and (\ref{dv2}) that%
\begin{align}
&  \dot{V}_{1}\left(  x_{t}\right)  +\dot{V}_{2}\left(  x_{t}\right)
\nonumber \\
&  \leq x^{\mathrm{T}}\left(  t\right)  \sum \limits_{i=1}^{N}\left(  \tau
_{i}^{2}A_{i}^{\mathrm{T}}Q_{i}^{-1}A_{i}+\delta \tau_{i}I_{n}\right)  x\left(
t\right)  -\sum \limits_{i=1}^{N}y_{i}-\mu \nonumber \\
&  \leq \left(  1-\varepsilon \right)  x^{\mathrm{T}}\left(  t\right)
Q^{-1}x\left(  t\right)  -\sum \limits_{i=1}^{N}y_{i}-\mu \nonumber \\
&  \leq \left(  1-\varepsilon \right)  \sum \limits_{i=1}^{N}\left(
x^{\mathrm{T}}\left(  t\right)  R_{i}x\left(  t\right)  -x\left(  t-\tau
_{i}\right)  ^{\mathrm{T}}R_{i}x\left(  t-\tau_{i}\right)  \right)
-\mu \nonumber \\
&  =\left(  1-\varepsilon \right)  \dot{V}_{1}\left(  x_{t}\right)  -{\delta
\sum \limits_{i=1}^{N}\int_{-\tau_{i}}^{0}\left \Vert x\left(  t+s\right)
\right \Vert ^{2}\mathrm{d}s}, \label{eq34}%
\end{align}
where {$\mu=\delta \sum \limits_{i=1}^{N}\int_{-\tau_{i}}^{0}\left \Vert x\left(
t+s\right)  \right \Vert ^{2}\mathrm{d}s$} and we have used (\ref{eq28}).
Therefore%
\begin{equation}
\dot{V}\left(  x_{t}\right)  \triangleq \varepsilon \dot{V}_{1}\left(
x_{t}\right)  +\dot{V}_{2}\left(  x_{t}\right)  \leq-{\delta \int_{-\tau}%
^{0}\left \Vert x\left(  t+s\right)  \right \Vert ^{2}\mathrm{d}s.} \label{eq31}%
\end{equation}
Finally, it is trivial to show that $V\left(  x_{t}\right)  $ satisfies
(\ref{eq64}). The conclusion
then follows from Lemma \ref{lm0}.

\textit{Proof of Item 2}. By a Schur complement, the LMI in (\ref{LMI}) is
satisfied if and only if{
\begin{equation}
\left[
\begin{BMAT}(@){c|ccc}{c|ccc} -Q & \tau_{1}QA_{1}^{\mathrm{T}} & \cdots & \tau_{N}QA_{N}^{\mathrm{T}}\\ \tau_{1}A_{1}Q & -Q_{1} & &  \\ \vdots & & \ddots & \\ \tau_{N}A_{N}Q & & & -Q_{N}\end{BMAT}\right]
<0. \label{eqxx}
\end{equation}}
By using a Schur complement again, the inequality (\ref{eqxx}) can be
equivalently transformed into
\begin{equation}
-Q+\sum \limits_{i=1}^{N}\tau_{i}^{2}QA_{i}^{\mathrm{T}}Q_{i}^{-1}A_{i}Q<0,
\label{eq29}%
\end{equation}
which is further equivalent to {(\ref{eq44})}. The proof is finished.
\end{proof}

At the end of this section, we will show that the stability condition
(\ref{LMI}) has a very interesting relationship with the stability condition
of the following IDS
\begin{equation}
x\left(  t\right)  =\sum \limits_{i=1}^{N}A_{i}x\left(  t-\tau_{i}\right)  ,
\label{idslaa}%
\end{equation}
which was originally studied in \cite{carvalho96laa} by using a Lyapunov
functional approach, where $A_{i}\in \mathbf{R}^{n\times n},i\in \mathbf{I}%
\left[  1,N\right]  ,$ are given matrices and $\tau_{i},i\in \mathbf{I}\left[
1,N\right]  ,$ are given scalars and are such that $0<\tau_{1}<\cdots<\tau
_{N}.$

\begin{lemma}
\cite{carvalho96laa} \label{lm4}The IDS (\ref{idslaa}) is exponentially stable
independent of the delays $\tau_{i},i\in \mathbf{I}\left[  1,N\right]  ,$ if
there exist $N$ positive definite matrix $Q_{i}\in \mathbf{R}^{n\times n}%
,i\in \mathbf{I}\left[  1,N\right]  ,$ such that the following LMI is satisfied%
\begin{equation}%
{\displaystyle \sum \limits_{i=1}^{N}}
\left[
\begin{array}
[c]{c}%
A_{1}\\
\vdots \\
A_{N}%
\end{array}
\right]  Q_{i}\left[
\begin{array}
[c]{c}%
A_{1}\\
\vdots \\
A_{N}%
\end{array}
\right]  ^{\mathrm{T}}-\left[
\begin{array}
[c]{ccc}%
Q_{1} &  & \\
& \ddots & \\
&  & Q_{N}%
\end{array}
\right]  <0. \label{lmilaa}%
\end{equation}

\end{lemma}

\begin{proof}
By recognizing the characteristic polynomial of the IDS (\ref{idslaa}) (see
Eq. (3.7) in \cite{carvalho96laa}), it is not hard to see that it is
exponentially stable if and only if the following IDS%
\begin{equation}
x\left(  t\right)  =\sum \limits_{i=1}^{N}A_{i}^{\mathrm{T}}x\left(  t-\tau
_{i}\right)  , \label{idslaa1}%
\end{equation}
is. Then by the results in Section 4 in \cite{carvalho96laa}, the IDS
(\ref{idslaa1}) is exponentially stable independent of the delays $\tau
_{i},i\in \mathbf{I}\left[  1,N\right]  ,$ if there exist $N$ positive definite
matrices $X_{i},i\in \mathbf{I}\left[  1,N\right]  ,$ such that the following
LMI is satisfied (see inequality (4.6) in \cite{carvalho96laa})%
\begin{equation}
\left[
\begin{array}
[c]{cccc}%
\Pi_{1} & -A_{1}X_{1}A_{2}^{\mathrm{T}} & \cdots & -A_{1}X_{1}A_{N}%
^{\mathrm{T}}\\
-A_{2}X_{1}A_{1}^{\mathrm{T}} & \Pi_{2} & \cdots & -A_{2}X_{1}A_{N}%
^{\mathrm{T}}\\
\vdots & \vdots & \ddots & \vdots \\
-A_{N}X_{1}A_{1}^{\mathrm{T}} & -A_{N}X_{1}A_{2}^{\mathrm{T}} & \cdots &
X_{N}-A_{N}X_{1}A_{N}^{\mathrm{T}}%
\end{array}
\right]  >0, \label{lmi3}%
\end{equation}
where $\Pi_{i}=X_{i}-A_{i}X_{1}A_{i}^{\mathrm{T}}-X_{i+1},i\in \mathbf{I}%
\left[  1,N-1\right]  $. As the above LMI implies $X_{i}>X_{i+1}%
,i\in \mathbf{I}\left[  1,N-1\right]  ,$ we can let $Q_{N}=X_{N}$ and
$Q_{i}=X_{i}-X_{i+1}>0,\;i\in \mathbf{I}\left[  1,N-1\right]  .\label{Qii}$
Then the LMI in (\ref{lmi3}) can be exactly rewritten as (\ref{lmilaa}).

\end{proof}

It is very interesting to notice that the LMI (\ref{LMI}) and the LMI
(\ref{lmilaa}) possess the very similar structures and the only difference is
that the former LMI is delay dependent and the later one is not. This
similarity may help us to understand the stability of these two classes of
IDSs (\ref{ids}) and (\ref{idslaa})

\section{\label{sec3}A Spectral Radius Based Condition and A Comparison}

\subsection{A Spectral Radius Based Stability Condition}

In this subsection, we will present a spectral radius based sufficient
condition for the exponential stability of the IDS (\ref{ids}) based on Lemma
\ref{lm1}, as indicated by the following theorem.

\begin{theorem}
\label{th2}
The following statements are equivalent.

\begin{enumerate}
\item[A.] There exist $N+1$ positive definite matrices $P,Q_{i}\in
\mathbf{R}^{n\times n},i\in \mathbf{I}\left[  1,N\right]  ,$ such that the
coupled LMIs in (\ref{LMIamc}) are satisfied.

\item[B.] There exist $N$ positive definite matrices $Q_{i}\in \mathbf{R}%
^{n\times n},i\in \mathbf{I}\left[  1,N\right]  ,$ such that the following
coupled LMIs are satisfied%
\begin{equation}
N\tau_{i}^{2}A_{i}^{\mathrm{T}}\left(  \sum \limits_{j=1}^{N}Q_{j}\right)
A_{i}-Q_{i}<0,\;i\in \mathbf{I}\left[  1,N\right]  . \label{eq35}%
\end{equation}

\item[C.] The following condition is met, where $\rho \left(  A\right)  $
denotes the spectral radius of a square matrix $A:$%
\begin{equation}
\rho \left(  \sum \limits_{i=1}^{N}\tau_{i}^{2}A_{i}\otimes A_{i}\right)
<\frac{1}{N}. \label{eq43}%
\end{equation}

\item[D.] There exists a positive definite matrix $Q\in \mathbf{R}^{n\times n}$
such that the following LMI is satisfied
\begin{equation}
\sum \limits_{i=1}^{N}N\tau_{i}^{2}A_{i}^{\mathrm{T}}QA_{i}-Q<0. \label{eq30}%
\end{equation}

\end{enumerate}
\end{theorem}

\begin{proof}
\textit{Proof of A}$\Leftrightarrow$\textit{B}. It is clear that\ the LMIs in
(\ref{LMIamc}) are satisfied if we choose $P=\varepsilon I_{n}$ with
$\varepsilon$ being sufficiently small and the following coupled LMIs%
\begin{equation}
N\tau_{i}A_{i}^{\mathrm{T}}\left(  \sum \limits_{j=1}^{N}\tau_{j}Q_{j}\right)
A_{i}-Q_{i}<0,\;i\in \mathbf{I}\left[  1,N\right]  , \label{eq60}%
\end{equation}
are satisfied. The converse is obvious. Finally, the LMIs (\ref{eq60}) are
equivalent to (\ref{eq35}) by the substitution $\tau_{i}%
Q_{i}\rightarrow Q_{i},i\in \mathbf{I}\left[  1,N\right]  .$

\textit{Proof of B}$\Leftrightarrow$\textit{C}. Let $\mathcal{S}_{+}^{n\times
n}=\left(  S_{1},S_{2},\cdots,S_{N}\right)  $ where $S_{i}\in \mathbf{S}%
_{+}^{n\times n}\triangleq \{S:S=S^{\mathrm{T}}>0\}.$ Then
\begin{equation}
\mathscr{L}\left(  \mathcal{Q}\right)  =\left(  N\tau_{1}^{2}A_{1}%
^{\mathrm{T}}\sum \limits_{j=1}^{N}Q_{j}A_{1},\cdots,N\tau_{N}^{2}%
A_{N}^{\mathrm{T}}\sum \limits_{j=1}^{N}Q_{j}A_{N}\right)  , \label{eq70}%
\end{equation}
where $\mathcal{Q}=\left(  Q_{1},\cdots,Q_{N}\right)  \in \mathcal{S}%
_{+}^{n\times n},$ is a linear positive operator (see Definition 1 in
\cite{lzwd11tac}). Consequently, the inequalities in (\ref{eq35}) are
satisfied if and only if there exists a $\mathcal{Q}\in \mathcal{S}%
_{+}^{n\times n}$ such that%
\begin{equation}
\mathscr{L}\left(  \mathcal{Q}\right)  -\mathcal{Q}<0, \label{eq50}%
\end{equation}
where $\mathcal{P}<0$ means $-\mathcal{P}\in \mathcal{S}_{+}^{n\times n}.$ Then
by Lemma 1 in \cite{lzlw11amc}, the inequality in (\ref{eq50}) has a solution
$\mathcal{Q}\in \mathcal{S}_{+}^{n\times n}$ if and only if $\rho \left(
\mathscr{L}\right)  <1.$ However, similar to (10)-(11) in \cite{lzwd11tac}, we
can show that $\rho \left(  \mathscr{L}\right)  =\rho \left(  NA^{\mathrm{T}%
}\right)  =\rho \left(  NA\right)  $ with
\[
A=\left[
\begin{array}
[c]{cccc}%
\tau_{1}^{2}A_{1}^{\mathrm{T}}\otimes A_{1}^{\mathrm{T}} & \tau_{1}^{2}%
A_{1}^{\mathrm{T}}\otimes A_{1}^{\mathrm{T}} & \cdots & \tau_{1}^{2}%
A_{1}^{\mathrm{T}}\otimes A_{1}^{\mathrm{T}}\\
\tau_{2}^{2}A_{2}^{\mathrm{T}}\otimes A_{2}^{\mathrm{T}} & \tau_{2}^{2}%
A_{2}^{\mathrm{T}}\otimes A_{2}^{\mathrm{T}} & \cdots & \tau_{2}^{2}%
A_{2}^{\mathrm{T}}\otimes A_{2}^{\mathrm{T}}\\
\vdots & \vdots & \ddots & \vdots \\
\tau_{N}^{2}A_{N}^{\mathrm{T}}\otimes A_{N}^{\mathrm{T}} & \tau_{N}^{2}%
A_{N}^{\mathrm{T}}\otimes A_{N}^{\mathrm{T}} & \cdots & \tau_{N}^{2}%
A_{N}^{\mathrm{T}}\otimes A_{N}^{\mathrm{T}}%
\end{array}
\right]  .
\]
Hence the LMIs in (\ref{eq35}) are solvable if and only if%
\begin{equation}
\rho \left(  A\right)  <\frac{1}{N}. \label{rho}%
\end{equation}
Now notice that we can write $A=BC$ where
\begin{equation}
B=\left[
\begin{array}
[c]{c}%
\tau_{1}^{2}A_{1}^{\mathrm{T}}\otimes A_{1}^{\mathrm{T}}\\
\vdots \\
\tau_{N}^{2}A_{N}^{\mathrm{T}}\otimes A_{N}^{\mathrm{T}}%
\end{array}
\right]  ,\;C=\left[
\begin{array}
[c]{cccc}%
I_{n^{2}} & I_{n^{2}} & \cdots & I_{n^{2}}%
\end{array}
\right]  . \label{bc}%
\end{equation}
On the other hand, for any two matrices $X$ and $Y$ with appropriate
dimensions, we have $\rho \left(  XY\right)  =\rho \left(  YX\right)  .$ Hence%
\begin{equation}
\rho \left(  A\right)  =\rho \left(  BC\right)  =\rho \left(  CB\right)
=\rho \left(  \sum \limits_{i=1}^{N}\tau_{i}^{2}A_{i}^{\mathrm{T}}\otimes
A_{i}^{\mathrm{T}}\right)  , \label{rhoab}%
\end{equation}
by which the inequality in (\ref{rho}) is exactly the one in (\ref{eq43}).

\textit{Proof of C}$\Leftrightarrow$\textit{D}. The proof is quite similar to
the proof of B$\Leftrightarrow$C by utilizing the linear positive operator%
\begin{equation}
\mathscr{F}\left(  Q\right)  =\sum \limits_{i=1}^{N}N\tau_{i}^{2}%
A_{i}^{\mathrm{T}}QA_{i}, \label{eq491}%
\end{equation}
where $Q\in \mathbf{S}_{+}^{n\times n}$. The proof is finished.
\end{proof}


Item C of Theorem \ref{th2} implies an interesting spectral radius based
sufficient condition for testing the stability of the IDS (\ref{ids}), as
highlighted in the following corollary.

\begin{corollary}
\label{coro2}The IDS (\ref{ids}) is exponentially stable if the spectral
radius condition (\ref{eq43}) is satisfied. Particularly, if $N=1,$ the IDS
(\ref{ids}) is exponentially stable if $\rho \left(  A_{1}\right)  <\frac
{1}{\tau_{1}}.\label{eq55}$
\end{corollary}

\begin{remark}
If $N=1,$ by Proposition 2 in \cite{Aguliar10AMC}, the IDS (\ref{ids}) is
exponentially stable if $\left \Vert A_{1}\right \Vert <\frac{1}{\tau_{1}},$
which is {more} conservative than $\rho \left(  A_{1}\right)  <\frac
{1}{\tau_{1}}$ since $\rho \left(  A_{1}\right)
\leq \left \Vert A_{1}\right \Vert $ for any
matrix $A_{1}.$
\end{remark}

\begin{corollary}
\label{coro4}{The IDS (\ref{ids}) is exponentially stable if there
exist $N$ scalars $\alpha_{i}\in \left(  0,1\right)  ,i\in \mathbf{I}\left[  1,N\right]
$ such that $\textstyle \sum \nolimits_{i=1}^{N}
\alpha_{i}=1$ and
\begin{equation}
\rho \left(  \sum \limits_{i=1}^{N}\frac{\tau_{i}^{2}}{\alpha_{i}}A_{i}\otimes
A_{i}\right)  <1. \label{eq74}%
\end{equation}}

\end{corollary}

\begin{proof}
{According to the proof of Theorem \ref{th2}, (\ref{eq74}) is satisfied
if and only if there exists a $Q>0$ such that
\begin{equation}
\sum \limits_{i=1}^{N}\frac{\tau_{i}^{2}}{\alpha_{i}}A_{i}^{\mathrm{T}}%
QA_{i}-Q<0. \label{eq75}%
\end{equation}
Hence the inequality (\ref{eq44}) is satisfied with
$Q_{i}=\alpha_{i}Q.$ The result then follows from Theorem \ref{th1}}.
\end{proof}

{Clearly, the spectral condition (\ref{eq74}) reduces to (\ref{eq43})
if we set $\alpha_{i}=\frac{1}{N},i\in \mathbf{I}\left[  1,N\right]  .$}


\subsection{A Comparison of Different Sufficient Conditions}

With the help of Theorem \ref{th2}, we are able to make a comparison among
these different stability conditions in Lemma \ref{lm1}, Theorem \ref{th3} and
Theorem \ref{th1}.

\begin{proposition}
\label{coro1}The following statements are true.

\begin{enumerate}
\item If the set of LMIs (\ref{LMIamc}) (or (\ref{eq35})) are solvable, then
the set of LMIs in (\ref{LMI1})--(\ref{LMI2}) are also solvable, namely,
Theorem \ref{th3} is always less conservative than Lemma \ref{lm1}.

\item {The set of LMIs in (\ref{LMI1})--(\ref{LMI2}) are solvable if
and only if the LMI in (\ref{LMI}) is solvable. Hence Theorem \ref{th1} is
equivalent to Theorem \ref{th3} and both of them are thus less conservative than Lemma
\ref{lm1}}.
\end{enumerate}
\end{proposition}

\begin{proof}
\textit{Proof of Item 1}. By Theorem \ref{th2}, the set of LMIs (\ref{LMIamc})
are solvable if and only if the set of LMIs (\ref{eq35}) are solvable, which
implies that there exists a sufficiently small number $\varepsilon>0$ such
that $\varepsilon I_{n}<Q_{i},i\in \mathbf{I}\left[  1,N\right]  ,$ and%
\begin{equation}
N\tau_{i}^{2}A_{i}^{\mathrm{T}}\left(  \sum \limits_{j=1}^{N}Q_{j}\right)
A_{i}-Q_{i}+\varepsilon I_{n}<0,\;i\in \mathbf{I}\left[  1,N\right]  .
\label{eq51}%
\end{equation}
Now for every $i\in \mathbf{I}\left[  1,N\right]  $ we choose%
\begin{equation}
P_{i}^{-1}=N%
{\displaystyle \sum \limits_{i=1}^{N}}
Q_{i}\triangleq NQ,\;R_{i}=Q_{i}-\varepsilon I_{n}>0, \label{eq57}%
\end{equation}
which implies%
\begin{equation}
\sum \limits_{i=1}^{N}P_{i}=\sum \limits_{i=1}^{N}\frac{1}{N}Q^{-1}=Q^{-1}.
\label{eq52}%
\end{equation}
Then it follows from (\ref{eq51}) that, for all $i\in \mathbf{I}\left[
1,N\right]$,
\begin{equation}
\tau_{i}^{2}A_{i}^{\mathrm{T}}P_{i}^{-1}A_{i}-R_{i}=N\tau_{i}^{2}%
A_{i}^{\mathrm{T}}QA_{i}-Q_{i}+\varepsilon I_{n}<0  , \label{eq53}%
\end{equation}
and it follows from (\ref{eq52}) that
\begin{equation}%
{\displaystyle \sum \limits_{i=1}^{N}}
R_{i}=%
{\displaystyle \sum \limits_{i=1}^{N}}
\left(  Q_{i}-\varepsilon I_{n}\right)  =Q-N\varepsilon I_{n}<Q=\left(
\sum \limits_{i=1}^{N}P_{i}\right)  ^{-1}. \label{eq56}%
\end{equation}
Clearly, (\ref{eq53}) and (\ref{eq56}) are respectively in the form of
(\ref{NMI1})--(\ref{NMI2}).

\textit{Proof of Item 2}. Assume that (\ref{NMI1})--(\ref{NMI2}) are feasible.
Summing the $N$ nonlinear matrix inequalities in (\ref{NMI1}) on both sides
gives%
\begin{equation}%
{\displaystyle \sum \limits_{i=1}^{N}}
\tau_{i}^{2}A_{i}^{\mathrm{T}}Q_{i}^{-1}A_{i}-%
{\displaystyle \sum \limits_{i=1}^{N}}
S_{i}<0, \label{eq62}%
\end{equation}
which, by using the nonlinear matrix inequality (\ref{NMI2}), implies%
\begin{equation}%
{\displaystyle \sum \limits_{i=1}^{N}}
\tau_{i}^{2}A_{i}^{\mathrm{T}}Q_{i}^{-1}A_{i}<\left(
{\displaystyle \sum \limits_{i=1}^{N}}
Q_{i}\right)  ^{-1}, \label{eq63}%
\end{equation}
which is just in the form of (\ref{eq44}) and is further equivalent to
(\ref{LMI}).

{Now assume that (\ref{eq44}) is feasible. Denote
\begin{equation}
\Omega=\frac{1}{2N}\left(  \left(
{\displaystyle \sum \limits_{i=1}^{N}}
Q_{i}\right)  ^{-1}-%
{\displaystyle \sum \limits_{i=1}^{N}}
\tau_{i}^{2}A_{i}^{\mathrm{T}}Q_{i}^{-1}A_{i}\right)  , \label{eq33}%
\end{equation}
and let
\begin{equation}
S_{i}=\tau_{i}^{2}A_{i}^{\mathrm{T}}Q_{i}^{-1}A_{i}+\Omega>0,\;i\in
\mathbf{I}\left[  1,N\right]  . \label{eq36}%
\end{equation}
It follows that (\ref{NMI1}) is satisfied. Now, by (\ref{eq44}), we
have
\[%
{\displaystyle \sum \limits_{i=1}^{N}}
S_{i}=\frac{1}{2}\left(  \left(
{\displaystyle \sum \limits_{i=1}^{N}}
Q_{i}\right)  ^{-1}+%
{\displaystyle \sum \limits_{i=1}^{N}}
\tau_{i}^{2}A_{i}^{\mathrm{T}}Q_{i}^{-1}A_{i}\right)  <\left(
{\displaystyle \sum \limits_{i=1}^{N}}
Q_{i}\right)  ^{-1}%
\]
which implies that (\ref{NMI2}) is satisfied. The proof is finished.}
\end{proof}

This proposition demonstrates in theory that the multiple Jensen inequality
(\ref{jensen}) used in the proofs of Theorems \ref{th3} and \ref{th1} can
indeed reduce the conservatism of the resulting stability conditions.

\begin{remark}
{Though Theorems \ref{th3} and \ref{th1} are equivalent by Proposition
\ref{coro1}, Theorem \ref{th1} obtained by the novel Lyapunov functionals
(\ref{eqv1}) and (\ref{v2}) possesses an advantage over Theorem \ref{th3}. To
see this, we notice that the total row size (denoted by $\Phi$) and
the total number of scalar decision variables (denoted by $\Psi$) in
the LMIs of Theorems \ref{th3} and \ref{th1} are, respectively, given by
\begin{equation}
\left \{
\begin{array}
[c]{ll}%
\Phi_{\mathrm{Th.1}}=n+2n,\; & \Psi_{\mathrm{Th.1}}=n\left(  n+1\right)
N+n^{2},\\
\Phi_{\mathrm{Th.2}}=nN,\; & \Psi_{\mathrm{Th.2}}=\frac{n\left(  n+1\right)
}{2}N.
\end{array}
\right.  \label{eq72}%
\end{equation}
It is well known that the computational complexity of an LMI is
bounded by $\mu \Phi \Psi^{3}$ where $\mu$ is a constant (see
\cite{zzd11auto}). Hence the computation complexity of the LMIs in Theorem
\ref{th1} is significantly lower than that in Theorem \ref{th3}}.
\end{remark}

Combining Lemma \ref{lm4}, Proposition \ref{coro1} and Theorem \ref{th2}
gives the following corollary.

\begin{corollary}
The IDS (\ref{idslaa}) is exponentially stable independent of the delays
$\tau_{i},i\in \mathbf{I}\left[  1,N\right]  ,$ if $\rho \left(
{\textstyle \sum \nolimits_{i=1}^{N}}
A_{i}\otimes A_{i}\right)  <\frac{1}{N}.$
\end{corollary}
\vspace{-0.3cm}
{\begin{center}
{\normalsize Table 1: The maximal allowable $\tau_{2}$ by using different
approaches\vspace{0.3cm}
\begin{tabular}
[c]{l|cccc}\hline \hline
& LMI (\ref{LMI}) & LMIs (\ref{LMIamc}) & LMI (\ref{eq30}) & Condition
(\ref{eq43})\\ \hline
$\tau_{1}=0.4$ & 0.0317 & infeasible & infeasible & infeasible\\
\multicolumn{1}{c|}{$\tau_{1}=0.3$} & 0.1146 & 0.0474 & 0.0474 & 0.0474\\
\multicolumn{1}{c|}{$\tau_{1}=0.2$} & 0.2418 & 0.1527 & 0.1527 & 0.1527\\
\multicolumn{1}{c|}{$\tau_{1}=0.1$} & 0.4882 & 0.3414 & 0.3414 &
0.3414\\ \hline \hline
\end{tabular}
}
\end{center}}

\section{\label{sec4}A Numerical Example}

We consider an IDS in the form of (\ref{ids}) with two delays, namely,%
\begin{equation}
x\left(  t\right)  =A_{1}\int_{-\tau_{1}}^{0}x\left(  t+s\right)
\mathrm{d}s+A_{2}\int_{-\tau_{2}}^{0}x\left(  t+s\right)  \mathrm{d}s,
\label{sys2}%
\end{equation}
where%
\begin{equation}
A_{1}=\left[
\begin{array}
[c]{cc}%
-4 & 1\\
-13 & 2
\end{array}
\right]  ,\;A_{2}=\left[
\begin{array}
[c]{cc}%
0 & -1\\
1 & 0
\end{array}
\right]  . \label{eqa1a2}%
\end{equation}
If $A_{2}=0,$ this system has been studied in \cite{Aguliar10AMC}. It is shown
there that the LMIs in Lemma \ref{lm1} are feasible if and only if $0\leq
\tau_{1}\leq0.4473=\tau_{1}^{\ast}.$ On the other hand, direct computation
gives $\rho \left(  \tau_{1}^{\ast}A_{1}\right)  =0.9999,$ which clearly
validates Corollary \ref{coro2}.

{For a fixed $\tau_{1},$ the maximal value of $\tau_{2}$ such that the LMI in
(\ref{LMI}), the LMIs in (\ref{LMIamc}), the LMI in (\ref{eq30}), and the
spectral condition (\ref{eq43}) are feasible can be respectively computed by a
bisection method. The results are recorded in Table 1. From this table we can
observe that Theorem \ref{th1} is always less
conservative than Lemma \ref{lm1} established in \cite{Aguliar10AMC}, which
indicates that our approach based on the multiple Jensen inequality can
considerably reduce the conservatism in the stability analysis of this class
of IDSs. Moreover, the LMIs (\ref{LMIamc}), the LMI (\ref{eq30}) and the
spectral condition (\ref{eq43}) lead to the same result, which validates
Theorem \ref{th2}.  We mention that the results obtained by Theorem \ref{th3} are the same as
that by Theorem \ref{th1}. However, from the computational point of view, Theorem \ref{th2} is recommended to use in practice as it only involves one
constraint and a single decision variable. }

{To illustrate Corollary \ref{coro4}, let $\tau_{1}=0.4$ and
$\tau_{2}=0.02.$ From Table 1 we can see that (\ref{eq43}) is not
satisfied. Since $\left \Vert A_{1}\right \Vert \gg \left \Vert A_{2}\right \Vert
,$ we may let the weighting factor $\frac{1}{\alpha_{1}}$ of
$\left(  \tau_{1}A_{1}\right)  \otimes \left(  \tau_{1}A_{1}\right)  $
be small enough so that the spectral radius of the resulting matrix $%
{\textstyle \sum \nolimits_{i=1}^{2}}
\frac{1}{\alpha_{i}}\left(  \tau_{i}A_{i}\right)  \otimes \left(  \tau_{i}%
A_{i}\right)  $ is less than one. Indeed, if we choose $\alpha
_{1}=0.9$ and $\alpha_{2}=0.1$ we can compute $\rho(%
{\textstyle \sum \nolimits_{i=1}^{2}}
\frac{1}{\alpha_{i}}\left(  \tau_{i}A_{i}\right)  \otimes \left(  \tau_{i}%
A_{i}\right)  )=0.9783,$ which implies the asymptotic stability
of IDS (\ref{sys2}) in this case.}

{We finally mention that the conclusions obtained in the above have
been approved by many other randomly chosen numerical examples that are not
included here to save spaces.}

\vspace{-0.15cm}
\section{\label{sec5}Conclusion}
This note has studied the stability analysis of a class of integral delay
systems (IDSs) with multiple delays, which have wide applications in the
stability analysis of time-delay systems, especially for neutral time-delay
systems. By generalizing the well-known Jensen inequality to the case with
multiple terms through introducing multiple weighting matrices, two Lyapunov
functional based approaches have been established to yield set of sufficient
stability conditions. Moreover, it is
shown by the positive operator theory that the obtained new conditions are
always less conservative than the existing ones and a spectral radius based
sufficient condition is obtained simultaneously. A numerical example has
demonstrated the effectiveness of the established approaches.

\end{document}